\documentclass[11pt,reqno]{amsart} 
\usepackage{amssymb,latexsym}
\usepackage{cite} 
\usepackage{amsmath}
\usepackage[height=190mm,width=130mm]{geometry} 
\theoremstyle{plain}
\newtheorem{theorem}{Theorem}
\newtheorem{lemma}{Lemma}

\newtheorem{proposition}{Proposition}

\theoremstyle{definition}

\theoremstyle{remark}

\numberwithin{equation}{section} 

\begin{document}
\title[Two divisors of $(n^2+1)/2$ summing up to $\delta n+\varepsilon$]{Two divisors of $(n^2+1)/2$ summing up to $\delta n+\varepsilon$, \newline for $\delta$ and $\varepsilon$ even} 
\author{Sanda Buja\v{c}i\'{c}}
\address{University of Rijeka\newline Department of Mathematics\newline Radmile Matej\v{c}i\'{c} 2\newline
51 000 Rijeka\newline Croatia}
\email{sbujacic@math.uniri.hr}

\begin{abstract}
In this paper we are dealing with the problem of the existence of two divisors of $(n^2+1)/2$ whose sum is equal to $\delta n+\varepsilon$, in the case when $\delta$ and $\varepsilon$ are even, or more precisely in the case in which $\delta\equiv\varepsilon+2\equiv0$ or $2 \pmod{4}$. We will completely solve the cases $\delta=2, \delta=4$ and $\varepsilon=0$.
\end{abstract}

\maketitle

\section{Introduction}
In \cite{ayad}, Ayad and Luca have proved that there does not exist an odd integer $n>1$ and two positive divisors $d_1, d_2$ of $\frac{n^2+1}{2}$ such that $d_1+d_2=n+1$. In \cite{dujella}, Dujella and Luca have dealt with a more general issue, where $n+1$ was replaced with an arbitrary linear polynomial $\delta n+\varepsilon$,  where $\delta>0$  and $\varepsilon$ are given integers. The reason that $d_1$ and $d_2$ are congruent to $1$ modulo $4$ comes from the fact that $(n^2+1)/2$ is odd and is a sum of two coprime squares $((n+1)/2)^2+((n-1)/2)^2$. Such numbers have the property that all their prime factors are congruent to $1$ modulo $4$.
Since $d_1+d_2=\delta n + \varepsilon$, then there are two cases: it is either $\delta\equiv\varepsilon\equiv1\pmod{2}$, or $\delta\equiv\varepsilon+2\equiv0 \enspace \textrm{or}\enspace 2\pmod{4}$. In \cite{dujella} authors have focused on the first case. 
\newline\indent In this paper, we deal with the second case, the case where $\delta\equiv\varepsilon+2\equiv0 \enspace \textrm{or}\enspace 2\pmod{4}$. We completely solve cases when $\delta=2, \delta=4$ and $\varepsilon=0$. We prove that there exist infinitely many positive odd integers $n$ with the property that there exists a pair of positive divisors $d_1, d_2$ of $\frac{n^2+1}{2}$ such that  $d_1+d_2=2n+\varepsilon$ for $\varepsilon\equiv0\pmod{4}$ and we prove an analoguos result for $\varepsilon\equiv2\pmod{4}$ and divisors $d_1, d_2$ of $\frac{n^2+1}{2}$ such that $d_1+d_2=4n+\varepsilon$. In case when $\delta\geq6$ is a positive integer of the form $\delta=4k+2,\enspace
k\in\mathbb{N}$  we prove that there does not exist an odd integer $n$ such that there exists a pair of divisors $d_1, d_2$ of $\frac{n^2+1}{2}$ with the property $d_1+d_2=\delta n$.
We also prove that there exist infinitely many odd integers $n$ with the property that there exists a pair of positive divisors $d_1, d_2$ of $\frac{n^2+1}{2}$ such that $d_1+d_2=2n$.

\begin{section}{The case $\delta=2$}
\begin{theorem}\label{theo:1}
If $\varepsilon\equiv0\pmod{4}$, then there exist infinitely many positive odd integers $n$ with the property that there exists a pair of positive divisors $d_1, d_2$ of $\frac{n^2+1}{2}$ such that $d_1+d_2=2n+\varepsilon$.
\end{theorem}
\begin{proof}
Let $\varepsilon\equiv0\pmod{4}$. We want to find a positive odd integer $n$ and positive divisors $d_1, d_2$ of $\frac{n^2+1}{2}$ such that $d_1+d_2=2n+\varepsilon$. 
Let $g=\textrm{gcd}(d_1, d_2)$. We can write $d_1=gd_1', d_2=gd_2'$. Since  $gd_1'd_2'=\textrm{lcm}(d_1, d_2)$ divides $\frac{n^2+1}{2}$, we conclude that there exists a positive integer $d$ such that 
\begin{equation*}
d_1d_2=\frac{g(n^2+1)}{2d}.
\end{equation*}
\indent From the identity 
\begin{equation*}
(d_2-d_1)^2=(d_1+d_2)^2-4d_1d_2,
\end{equation*} we can easily obtain \begin{equation*}(d_2-d_1)^2=(2n+\varepsilon)^2-4\frac{g(n^2+1)}{2d},\end{equation*}
\begin{equation*}(d_2-d_1)^2=4n^2+4\varepsilon n+\varepsilon^2-2\frac{g(n^2+1)}{d},\end{equation*}
\begin{equation*}d(d_2-d_1)^2=4n^2d+4d\varepsilon n+\varepsilon^2d-2n^2g-2g,\end{equation*}
\begin{equation*}d(d_2-d_1)^2=(4d-2g)n^2+4d\varepsilon n+\varepsilon^2d-2g,\end{equation*}
\begin{equation}\label{eq_solve22}
 d(4d-2g)(d_2-d_1)^2=(4d-2g)^2n^2+4(4d-2g)d\varepsilon n+4d^2\varepsilon^2-8dg-2\varepsilon^2dg+4g^2.
\end{equation}
\indent For $X=(4d-2g)n+2d\varepsilon, Y=d_2-d_1,$ the equation $(\ref{eq_solve22})$ becomes
\begin{equation*}X^2-d(4d-2g)Y^2=8dg+2\varepsilon^2dg-4g^2.\end{equation*}
For $g=1$ the previous equation becomes \begin{equation*}X^2-2d(2d-1)Y^2=8d+2\varepsilon^2d-4,\end{equation*}
\begin{equation}\label{eq:solve}
X^2-2d(2d-1)Y^2=2d(4+\varepsilon^2)-4.\end{equation}
The equation $(\ref{eq:solve})$ is a Pellian equation. The right-hand side of $(\ref{eq:solve})$ is nonzero.\newline\newline
\indent Our goal is to make the right-hand side of $(\ref{eq:solve})$ a perfect square. That condition can be satisfied by taking $d=\frac{1}{8}\varepsilon^2-\frac{1}{2}\varepsilon+1$. With this choice of $d$, we get
\begin{equation*}
2d(4+\varepsilon^2)-4=2\left(\frac{1}{8}\varepsilon^2-\frac{1}{2}\varepsilon+1\right)(4+\varepsilon^2)-4=\left(\frac{1}{2}(\varepsilon^2-2\varepsilon+4)\right)^2.
\end{equation*}
\indent Pellian equation $(\ref{eq:solve})$ becomes \begin{equation}\label{eq:solve1}X^2-2d(2d-1)Y^2=\left(\frac{1}{2}(\varepsilon^2-2\varepsilon+4)\right)^2.\end{equation}
Now, like in \cite{dujella}, we are trying to solve $(\ref{eq:solve1})$. We let \begin{equation*}X=\frac{1}{2}(\varepsilon^2-2\varepsilon+4)U,\enspace Y=\frac{1}{2}(\varepsilon^2-2\varepsilon+4)V.\end{equation*}
\noindent The equation $(\ref{eq:solve1})$ becomes
\begin{equation}\label{eq:solve2}
U^2-2d(2d-1)V^2=1.
\end{equation}
\indent Equation $(\ref{eq:solve2})$ is a Pell equation  which has infinitely many positive integer solutions $(U, V)$, and consequently, there exist infinitely many positive integer solutions $(X, Y)$ of $(\ref{eq:solve1})$.
\noindent The least positive integer solution  of $(\ref{eq:solve2})$ can be found using the continued fraction expansion of number $\sqrt{2d(2d-1)}$. \newline\indent We can easily get $\sqrt{2d(2d-1)}=[2d-1;\overline{2, 4d-2}]$. All positive solutions of $(\ref{eq:solve2})$ are given by $(U_m, V_m)$ for some $m\geq0$. The first few solutions are \newline $(U_0, V_0) = (1,0),$ \newline $(U_1, V_1)=(4d-1,2)$,\newline $(U_2, V_2)=(32d^2-16d+1, 16d-4)$, \newline
$(U_3, V_3)=(256d^3-192d^2+36d-1, 128d^2-64d+6),\dots$.\newline\newline
\indent Generally, solutions of $(\ref{eq:solve2})$ are generated by recursive expressions 
\begin{equation*}
U_0=1, \enspace U_1=4d-1,\enspace U_{m+2}=2(4d-1)U_{m+1}-U_m,\end{equation*}
\begin{equation}\label{eq:solve0}V_0=0,\enspace V_1=2,\enspace V_{m+2}=2(4d-1)V_{m+1}-V_m,\enspace m\in\mathbb{N}_0.
\end{equation}
\indent By induction on $m$, one gets that $U_m\equiv1\pmod{(4d-2)}, m\geq0$. Indeed, $U_0=1 \equiv 1\pmod{(4d-2)}, \enspace U_1=4d-1\equiv1\pmod{(4d-2)}$. Assume that $U_m\equiv U_{m-1}\equiv 1\pmod{(4d-2)}$. For $U_{m+1}$ we get \begin{center}
$U_{m+1}=2(4d-1)U_m-U_{m-1}\equiv2-1\equiv1\pmod{(4d-2)}.$\end{center}
Now, it remains to compute the corresponding values of $n$ which arise from \newline $X=(4d-2)n+2d\varepsilon$ and $X=\frac{1}{2}(\varepsilon^2-2\varepsilon+4)U$. We obtain \begin{equation*}n=\frac{\frac{1}{2}(\varepsilon^2-2\varepsilon+4)U-2d\varepsilon}{4d-2}.\end{equation*}
\indent  We want the above number $n$ to be a positive integer.\newline\newline
From $d=\frac{1}{8}\varepsilon^2-\frac{1}{2}\varepsilon+1$, it follows $4d-2=\frac{1}{2}\varepsilon^2-2\varepsilon+2$.
Note that $\varepsilon$ is even. So, congruences
\begin{equation*}
\frac{1}{2}(\varepsilon^2-2\varepsilon+4)U-2d\varepsilon \equiv 4d+\varepsilon-2-2d\varepsilon\equiv -(2d-1)\varepsilon \equiv 0\pmod{(4d-2)},
\end{equation*}
show us that all numbers $n$ generated in the specified way are integers.\newline\newline
\indent The first few values of number $n$, which we get from $U_1, U_2, U_3$, are

 \[
 \left\{ 
  \begin{array}{l l}
   n=\frac{1}{2}(\varepsilon^2-3\varepsilon+6),\\
   d_1=1,\\
   d_2=\varepsilon^2-2\varepsilon+5.\\
  \end{array} \right.
\]
\\
 \[
 \left\{ 
  \begin{array}{l l}
   n=\frac{1}{2}(\varepsilon^4-6\varepsilon^3+20\varepsilon^2-33\varepsilon+34),\\
   d_1=\varepsilon^2-2\varepsilon+5,\\
   d_2=\varepsilon^4-6\varepsilon^3+19\varepsilon^2-30\varepsilon+29.\\
  \end{array} \right.
\]
\\
 \[
 \left\{ 
  \begin{array}{l l}
   n=\frac{1}{2}(\varepsilon^6-10\varepsilon^5+50\varepsilon^4-148\varepsilon^3+281\varepsilon^2-323\varepsilon+198),\\
   d_1=\varepsilon^4-6\varepsilon^3+19\varepsilon^2-30\varepsilon+29,\\
   d_2=\varepsilon^6-10\varepsilon^5+49\varepsilon^4-142\varepsilon^3+262\varepsilon^2-292\varepsilon+169.\\
  \end{array} \right.
\]
\end{proof}
\end{section}
\begin{section}{The case $\delta=4$}
\begin{theorem}\label{theo:2}
If $\varepsilon\equiv2\pmod{4}$, then there exist infinitely many positive odd integers $n$ with the property that there exists a pair of positive divisors $d_1, d_2$ of $\frac{n^2+1}{2}$ such that $d_1+d_2=4n+\varepsilon$.
\end{theorem}
\begin{proof} Proof of this theorem will be slightly different from the proof of Theorem \ref{theo:1}. Instead of assuming that $\varepsilon\equiv2\pmod{4}$, we will distiguish two cases: in one case we will be dealing with $\varepsilon\equiv6\pmod{8}$ and we will apply strategies from \cite{dujella} and in the other case we will be dealing with $\varepsilon\equiv2\pmod{8}$ and we will use different methods in obtaining results.\newline\newline
We start with the case when $\varepsilon\equiv6\pmod{8}$.
We want to find odd positive integers $n$ and positive divisors $d_1, d_2$ of $\frac{n^2+1}{2}$ such that $d_1+d_2=4n+\varepsilon$.\newline\newline
\indent Let $g=\textrm{gcd}(d_1, d_2)$, $d_1=gd_1', d_2=gd_2'$ and  $d$ is a positive integer which satisfies the equation \begin{equation*}
d_1d_2=\frac{g(n^2+1)}{2d}.
\end{equation*}
From the identity 
\begin{equation*}
(d_2-d_1)^2=(d_1+d_2)^2-4d_1d_2,
\end{equation*}
we obtain 
\begin{equation*}
(d_2-d_1)^2=(4n+\varepsilon)^2-4\frac{g(n^2+1)}{2d},
\end{equation*}
\begin{equation*}
d(d_2-d_1)^2=(16d-2g)n^2+8d\varepsilon n+\varepsilon^2d-2g,
\end{equation*}
\begin{equation}\label{eq:solve3}
 d(16d-2g)(d_2-d_1)^2=(16d-2g)^2n^2+8(16d-2g)d\varepsilon n+16d^2\varepsilon^2-32dg-2\varepsilon^2dg+4g^2.
\end{equation}
\indent Let $X=(16d-2g)n+4d\varepsilon, \enspace Y=d_2-d_1.$ Equation $(\ref{eq:solve3})$ becomes
\begin{equation}\label{eq:solve0}
X^2-2d(8d-g)Y^2=32dg+2\varepsilon^2dg-4g^2.
\end{equation}
For $g=1$ the previous expression becomes 
\begin{equation}\label{eq:solve4}
X^2-2d(8d-1)Y^2=2d(16+\varepsilon^2)-4.
\end{equation}
It is obvious that $(\ref{eq:solve4})$ is a Pellian equation. The right-hand side of  $(\ref{eq:solve4})$ is nonzero.\newline\newline
\indent Our goal is to make the right-hand side of $(\ref{eq:solve4})$ a perfect square. That condition can be satisfied by taking $d=\frac{1}{32}\varepsilon^2-\frac{1}{8}\varepsilon+\frac{5}{8}$. With this choice for $d$, we get
\begin{equation*}
2d(16+\varepsilon^2)-4=2\left(\frac{1}{32}\varepsilon^2-\frac{1}{8}\varepsilon+\frac{5}{8}\right)(16+\varepsilon^2)-4=\left(\frac{1}{4}(\varepsilon^2-2\varepsilon+16)\right)^2.
\end{equation*}
So, Pellian equation $(\ref{eq:solve4})$ becomes \begin{equation}\label{equation1}
X^2-2d(8d-1)Y^2=\left(\frac{1}{4}(\varepsilon^2-2\varepsilon+16)\right)^2.\end{equation}
Let  
\begin{equation*}
X=\frac{1}{4}(\varepsilon^2-2\varepsilon+16)W, \enspace Y=\frac{1}{4}(\varepsilon^2-2\varepsilon+16)Z.
\end{equation*}
The equation $(\ref{equation1})$ becomes \begin{equation}\label{eq:solve5}
W^2-2d(8d-1)Z^2=1.
\end{equation}
\indent The equation (\ref{eq:solve5}) is a Pell equation which has infinitely many positive integer solutions $(W, Z)$, and consequently, there exist infinitely many positive integer solutions $(X, Y)$ of $(\ref{equation1})$.
\noindent The least positive integer solution of $(\ref{eq:solve5})$ can be found using the continued fraction expansion of number $\sqrt{2d(8d-1)}$. \newline\indent We can easily get 
\begin{equation*}
\sqrt{2d(8d-1)}=[4d-1;\overline{1, 2, 1, 8d-2}].
\end{equation*}
All positive solutions of $(\ref{eq:solve5})$ are given by $(W_m, Z_m)$ for some $m\geq0$. The first few solutions are\newline $(W_0, Z_0) = (1,0),$ \newline $(W_1, Z_1) = (16d-1,4)$,\newline $(W_2, Z_2) = (512d^2-64d+1, 128d-8), \dots$.\newline\indent Generally, solutions of $(\ref{eq:solve5})$ are generated by recursive expressions \begin{equation*}
W_0=1, \enspace W_1=16d-1,\enspace W_{m+2}=2(16d-1)W_{m+1}-W_m,
\end{equation*}
\begin{equation*}
Z_0=0,\enspace Z_1=4,\enspace Z_{m+2}=2(16d-1)Z_{m+1}-Z_m,\enspace m\in\mathbb{N}_0.
\end{equation*}
\indent By induction on $m$, one gets that $W_m\equiv1\pmod{(16d-2)}, m\geq0$. Indeed, $W_0=1 \equiv 1\pmod{(16d-2)}$, $W_1=16d-1\equiv1\pmod{(16d-2)}$.
Assume that $W_{m}\equiv W_{m-1}\equiv 1\pmod{(16d-2)}$. For $W_{m+1}$ we get
\begin{equation*}
W_{m+1}=2(16d-1)W_m-W_{m-1}\equiv 2-1 \equiv1 \pmod{(16d-2)}.
\end{equation*}
\indent Now, it remains to compute the corresponding values of $n$ which arise from\newline $X=(16d-2)n+4d\varepsilon$ and $X=\frac{1}{4}(\varepsilon^2-2\varepsilon+16)W$. We obtain \begin{equation*}
n=\frac{\frac{1}{4}(\varepsilon^2-2\varepsilon+16)W-4d\varepsilon}{16d-2}.
\end{equation*}
\indent We want to prove that number $n$ is a positive integer.\newline\newline
From $d=\frac{1}{32}\varepsilon^2-\frac{1}{8}\varepsilon+\frac{5}{8}$,it follows $8d-1=\frac{1}{4}\varepsilon^2-\varepsilon+4$. Number $\frac{\varepsilon}{2}$ is an odd integer. Thus, the congruences  
\begin{equation*}
\frac{1}{4}(\varepsilon^2-2\varepsilon+16)W-4d\varepsilon\equiv 8d-1+\frac{\varepsilon}{2}-4d\varepsilon \equiv(8d-1)(1-\frac{\varepsilon}{2})\equiv0\pmod{(16d-2)}
\end{equation*}
\noindent show us that all numbers $n$ generated in the specified way are integers.
\newline\newline\indent The first few values of number $n$, which we get from $W_1, W_2, W_3$, are
 \[
 \left\{ 
  \begin{array}{l l}
   n=\frac{1}{4}(\varepsilon^2-3\varepsilon+18),\\
   d_1=1\\
   d_2=\varepsilon^2-2\varepsilon+17.\\
  \end{array} \right.
\]
\\
 \[
 \left\{ 
  \begin{array}{l l}
   n=\frac{1}{4}(\varepsilon^4-6\varepsilon^3+44\varepsilon^2-105\varepsilon+322),\\
   d_1=\varepsilon^2-2\varepsilon+17,\\
   d_2=\varepsilon^4-6\varepsilon^3+43\varepsilon^2-102\varepsilon+305.\\
  \end{array} \right.
\]
\\
 \[
 \left\{ 
  \begin{array}{l l}
   n=\frac{1}{4}(\varepsilon^6-10\varepsilon^5+86\varepsilon^4-388\varepsilon^3+1529\varepsilon^2-3155\varepsilon+5778),\\
   d_1=\varepsilon^4-6\varepsilon^3+43\varepsilon^2-102\varepsilon+305,\\
   d_2=\varepsilon^6-10\varepsilon^5+85\varepsilon^4-382\varepsilon^3+1486\varepsilon^2-3052\varepsilon+5473.\\
  \end{array} \right.
\]\newline\newline
\indent Now, we deal with the case when $\varepsilon\equiv2\pmod{8}$. Let $\varepsilon=8k+2, \enspace k\in\mathbb{N}_0$. For $g=\frac{1}{4}\varepsilon^2+4$ and $g=d_1$, the equation $(\ref{eq:solve0})$ becomes
\begin{equation*}
X^2-2d(8d-g)Y^2=\frac{2d-1}{4}\varepsilon^4+8\varepsilon^2(2d-1)+64(2d-1).
\end{equation*}
The right-hand side of the equation will be a perfect square if $2d-1$ is a perfect square. Motivated by the experimental data, we take 
\begin{equation*}
d=\frac{1}{512}\varepsilon^4-\frac{1}{64}\varepsilon^3+\frac{7}{64}\varepsilon^2-\frac{5}{16}\varepsilon+\frac{41}{32}.
\end{equation*}
We get
\begin{equation*}
2d-1=16k^4+8k^2+1=(4k^2+1)^2.
\end{equation*}
So, the equation $(\ref{eq:solve0})$ becomes
\begin{equation}\label{eq:02}
X^2-2d(8d-g)Y^2=\left(\frac{1}{32}(\varepsilon^2+16)(\varepsilon^2-4\varepsilon+20)\right)^2.
\end{equation}
We consider the corresponding Pell equation
\begin{equation}\label{eq:03}
U^2-2d(8d-g)V^2=1.
\end{equation}
Let $(U_0, V_0)$ be the least positive integer solution of $(\ref{eq:03})$. That equation has infinitely many solutions. From $(\ref{eq:03})$ we get that 
\begin{equation*}
U^2\equiv1\pmod{(16d-2g)}.
\end{equation*} 
We deal with the case where $g=d_1=\frac{1}{4}\varepsilon^2+4$ and from the experimental data we can set 
\begin{equation*}
d_2=d_1^2-16kd_1, \enspace k\in\mathbb{N}_0.\end{equation*}
For $Y=d_2-d_1$ we get 
\begin{equation*}
Y=\left(\frac{1}{4}\varepsilon^2+4\right)^2-(2\varepsilon-3)\left(\frac{1}{4}\varepsilon^2+4\right)=\frac{\varepsilon^4}{16}-\frac{\varepsilon^3}{2}+\frac{11\varepsilon^2}{4}-8\varepsilon+28.
\end{equation*}
From $(\ref{eq:02})$, we obtain:
\begin{equation*}
X=\frac{(\varepsilon^2+16)(\varepsilon^6-16\varepsilon^5+140\varepsilon^4-768\varepsilon^3+3120\varepsilon^2-8704\varepsilon+14400)}{2048}.
\end{equation*}
We claim that $X$ satisfies the congruence \begin{equation}\label{eq:solvecon}
X\equiv4d\varepsilon\pmod{(16d-2g)}.
\end{equation} 
Indeed,
\begin{equation*}
16d-2g=\frac{\varepsilon^4}{32}-\frac{\varepsilon^3}{4}+\frac{5\varepsilon^2}{4}-5\varepsilon+\frac{25}{2},
\end{equation*}
\begin{equation*}
X-4d\varepsilon=\left(\frac{\varepsilon^4}{32}-\frac{\varepsilon^3}{4}+\frac{5\varepsilon^2}{4}-5\varepsilon+\frac{25}{2}\right)\left(\frac{\varepsilon^4}{64}-\frac{\varepsilon^3}{8}+\frac{13\varepsilon^2}{16}-\frac{9\varepsilon}{4}+9\right).\end{equation*}
From $n=\frac{X-4d\varepsilon}{16d-2g}$, we get 
\begin{equation*}
n=\frac{\varepsilon^4}{64}-\frac{\varepsilon^3}{8}+\frac{13\varepsilon^2}{16}-\frac{9\varepsilon}{4}+9=64k^4+28k^2+7,
\end{equation*}
and we see that $n$ is an odd integer.
Thus, if we define 
\begin{equation*}
\left(X_0, Y_0\right)
=\bigg(\frac{(\varepsilon^2+16)(\varepsilon^6-16\varepsilon^5+140\varepsilon^4-768\varepsilon^3+3120\varepsilon^2-8704\varepsilon+14400)}{2048},
\end{equation*}
\begin{equation*}
\frac{1}{16}(\varepsilon^2+16)(\varepsilon^2-8\varepsilon+28) \bigg),
\end{equation*}
we see that $(X_0, Y_0)$ is a solution of $(\ref{eq:02})$ which satisfies the congruence $(\ref{eq:solvecon})$.
\noindent We have proved that for every $\varepsilon\equiv2\pmod{8}$ there exists at least one odd integer $n$ which satisfies the conditions of Theorem \ref{theo:2}. Our goal is to prove that there exist infinitely many such integers $n$ that satisfy the properties of Theorem \ref{theo:2}. \newline\newline
If $(X_0, Y_0)$ is a solution of $(\ref{eq:02})$, solutions of $(\ref{eq:02})$ are also
\begin{equation}\label{eq:06}
(X_i, Y_i)=\left(X_0+\sqrt{2d(8d-g)}Y_0\right)\left(U_0+\sqrt{2d(8d-g)}V_0\right)^{2i}, \enspace i=0, 1, 2, \dots
\end{equation} 
From the equation $(\ref{eq:06})$, we get
\begin{equation*}
X_i\equiv U_0^{2i}X_0\equiv X_0\equiv4d\varepsilon\pmod{(16d-2g)}.
\end{equation*}
So, there are infinitely many solutions $(X_i, Y_i)$ of $(\ref{eq:02})$ that satisfy the congruence $(\ref{eq:solvecon})$. Therefore, by 
\begin{equation*}
n=\frac{X_i-4d\varepsilon}{16d-2g},
\end{equation*}
we get infinitely many integers $n$ with the required properties.
It is easy to see that number $n$ defined in this way
is odd. Indeed, we have $16d-2g\equiv2 \pmod{4}$,
$X_0\equiv2 \pmod{4}$, and since $(\ref{eq:03})$ implies that $U_0$ is odd and $V_0$ is even,
we get from $(\ref{eq:solvecon})$ that
\begin{equation*}
X_i-4d\varepsilon\equiv X_i\equiv U_0^{2i}X_0 \equiv X_0\equiv2 \pmod{4},
\end{equation*}
so $n$ is odd.
\end{proof}
\end{section}

\begin{section}{The case $\varepsilon=0$}
\begin{proposition}\label{prop:1}
There exist infinitely many positive odd integers $n$ with the property that there exists a pair of positive divisors $d_1, d_2$ of $\frac{n^2+1}{2}$ such that $d_1+d_2=2n$. These solutions satisfy $\textnormal{gcd}(d_1, d_2)=1$ \textrm{and} $d_1d_2=\frac{n^2+1}{2}$.
\end{proposition}
\begin{proof}
 We want to find a positive odd integer $n$ and positive divisors $d_1, d_2$ of $\frac{n^2+1}{2}$ such that $d_1+d_2=2n$. Let $g=\textrm{gcd}(d_1, d_2)$.  Then $ g|(2n)$ and $g|(n^2+1)$ which implies that $g|((2n)^2+4)$ so we can conclude that $g|4$. Because $g$ is the greatest common divisor of $d_1, d_2$ and $d_1, d_2$ are odd numbers, we can also conclude that $g$ is an odd number. So, $g=1$. Like we did in the proofs of the previous theorems, we define a positive integer $d$ which satisfies the equation $d_1d_2=\frac{n^2+1}{2d}.$ From the identity
\begin{equation*}
(d_2-d_1)^2=(d_1+d_2)^2-4d_1d_2,
\end{equation*}
we can easily obtain
\begin{equation*}
(d_2-d_1)^2=(2n)^2-2\frac{(n^2+1)}{d},
\end{equation*}
\begin{equation*}
d(d_2-d_1)^2=4n^2d-2n^2-2.
\end{equation*}
Let $d_2-d_1=2y$, so we get
\begin{equation*}
(4d-2)n^2-4dy^2=2,
\end{equation*}
\begin{equation}\label{eq:solve6}
(2d-1)n^2-2dy^2=1.
\end{equation}
\indent We will use the next lemma, which is Criterion 1 from \cite{grelak} to check if there exists a solution for (\ref{eq:solve6}).
\begin{lemma}\label{lemma}
Let $a>1$, $b$ be positive integers such that $\textrm{gcd}(a, b)=1$ and $D=ab$ is not a perfect square. Moreover, let $(u_0, v_0)$ denote the least positive integer solution of the Pell equation 
\begin{equation*}
u^2-Dv^2=1.
\end{equation*}
\noindent Then equation $ax^2-by^2=1$ has a solution in positive integers $x, y$ if and only if 
\begin{equation*}2a|(u_0+1) \enspace \textrm{and} \enspace 2b|(u_0-1).\end{equation*}\end{lemma}
\begin{flushright}
$\square$
\end{flushright}
We want to solve the Pell equation
\begin{equation}\label{eq:solve7}
U^2-2d(2d-1)V^2=1,
\end{equation}
where $n=U,\enspace y=V$. The continued fraction expansion of the number $\sqrt{2d(2d-1)}$ is already known from Theorem $\ref{theo:1}$ where we have obtained 
\begin{equation*}
\sqrt{2d(2d-1)}=[2d-1; \overline{2, 4d-2}].
\end{equation*} 
\indent The least positive integer solution of the Pell equation $(\ref{eq:solve7})$ is $(4d-1, 2)$. In our case, we want to find solutions of (\ref{eq:solve6}), so we apply Lemma \ref{lemma} which gives us conditions that have to be fulfilled. It has to be that
\begin{equation*}
2(2d-1)|4d \enspace\textrm{and} \enspace 4d|(4d-2),
\end{equation*} 
which is not true for $d\in\mathbb{N}$. So, for Pellian equation $(\ref{eq:solve6})$
there are no integer solutions $(n, y)$ when $a=2d-1>1$. Finally, we have to check the remaining case for $a=1$, which is the case that is not included in Lemma \ref{lemma}.\newline\newline
\indent If $a=2d-1=1$, then $d=1$.  From $(\ref{eq:solve6})$  and $d=1$, we get the Pell equation
\begin{equation}\label{eq:solve8}
n^2-2y^2=1,
\end{equation}
which has infinitely many solutions $n=U_m,\enspace y=V_m,\enspace m\in\mathbb{N}_0$ where
\begin{equation*}
U_0=1,\enspace U_1=3,\enspace U_{m+2}=6 U_{m+1}-U_m,
\end{equation*}
\begin{equation*}
V_0=0,\enspace V_1=2,\enspace V_{m+2}=6 V_{m+1}-V_m,\enspace m\in\mathbb{N}_0.
\end{equation*}
\indent The first few values $(U_i, V_i)$ are\newline
\newline $(U_0, V_0) = (1,0),$  $(U_1, V_1)=(3,2)$, $(U_2, V_2)=(17, 12)$,  $(U_3, V_3)=(99, 70), \dots$.\newline\newline
\indent From those solutions we can easily generate $(n, d_1, d_2)$
\begin{equation*}
(n, d_1, d_2)=(3, 1, 5),\enspace (17, 5, 29),\enspace (99, 29, 169), \dots.\end{equation*}
\indent  We have proved that there exist infinitely many odd positive integers $n$ with the property that there exists a pair of positive divisors $d_1, d_2$ of $\frac{n^2+1}{2}$ such that $d_1+d_2=2n$. We have also proved that $g=1$ and $d=1$,  so we conclude that numbers $d_1$ and $d_2$ are coprime and that $d_1d_2=\frac{n^2+1}{2}$.
\end{proof}
\begin{theorem}\label{theo:3}
Let $\delta\geq6$ be a positive integer such that $\delta=4k+2,
k\in\mathbb{N}$. Then there does not exist a positive odd integer $n$ with the property that there exists a pair of positive divisors $d_1, d_2$ of $\frac{n^2+1}{2}$ such that $d_1+d_2=\delta n$.
\end{theorem}
\begin{proof}
Suppose on the contrary that this is not so and let the number $\delta$ be the smallest positive integer $\delta=4k+2,\enspace k\in\mathbb{N}$ for which there exists an odd integer $n$ and a pair of positive divisors $d_1, d_2$ of $\frac{n^2+1}{2}$ such that $d_1+d_2=\delta n$. Let $g=\textrm{gcd}(d_1, d_2)>1$. Since $d_1=gd_1', \enspace d_2=gd_2'$, it follows that  $g|(n^2+1)$ and $g|(\delta n)$ and we conclude that $g|((\delta n)^2+\delta^2)$, which implies that $g|\delta^2$. This means that $g$ and $\delta$ have a common prime factor $p$. Let $d_1=pd_1'', d_2=pd_2'', \delta=p\delta''$. Then, we have $pd_1''+pd_2''=p\delta'' n,$ so we can conclude $d_1''+d_2''=\delta'' n$ where $d_1'', d_2''$ are divisors of $\frac{n^2+1}{2}$. It is clear that $\delta''<\delta$ and if it also satisfies $\delta''\neq2$, the existence of the number $\delta''$ contradicts the minimality of $\delta$. So, if $\delta''\neq2$, then we must have $g=1$.\newline\indent  If $\delta''=2$, it follows from Proposition \ref{prop:1} that $\textrm{gcd}(d_1'', d_2'')=1$ and $d_1''d_2''=\frac{n^2+1}{2}$. But, $\textrm{gcd}(d_1, d_2)=pd_1''d_2''$ should be a divisor of $\frac{n^2+1}{2}$ which is not possible because $p>1$. So, in this case we also conclude that $g=1$.\newline
\indent From the identity 
\begin{equation*}
(d_2-d_1)^2=(d_1+d_2)^2-4d_1d_2,
\end{equation*}
and using $g=1$, we obtain
\begin{equation*}
(d_2-d_1)^2=(\delta n)^2-2\frac{(n^2+1)}{d},
\end{equation*}
\begin{equation*}
d(d_2-d_1)^2=\delta^2n^2d-2n^2-2,
\end{equation*}
\begin{equation*}
d(d_2-d_1)^2=(d\delta^2-2)n^2-2.
\end{equation*}
In the equation 
\begin{equation*}
(\delta^2d-2)n^2-d(d_2-d_1)^2=2,
\end{equation*}
we set $(d_2-d_1)=2y$ (number $d_2-d_1$ is an even number because $d_1, d_2$ are odd integers), and we get \begin{equation*}
(\delta^2d-2)n^2-4dy^2=2.
\end{equation*}
\indent If we divide both sides of the above equation by $2$, then it becomes 
\begin{equation*}
(2d(2k+1)^2-1)n^2-2dy^2=1.
\end{equation*}
Now, if we define $\delta'=\frac{\delta}{2}=2k+1$, we get
\begin{equation}\label{eq:solve9}
(2\delta'^2d-1)n^2-2dy^2=1.
\end{equation}
We will prove by applying Lemma \ref{lemma} that the above Pell equation $(\ref{eq:solve9})$ has no solutions. \newline\newline
\indent To be able to apply Lemma \ref{lemma}, we have to deal with an equation of the form 
\begin{equation*}
x^2-Dy^2=1.
\end{equation*}
We have $a=2d\delta'^2-1$, $a>1$ (because $\delta'\geq3$) and $D=ab=2d(2\delta'^2d-1)$ is not a perfect square because $2d(2\delta'^2d-1)\equiv2\pmod{4}$. We need to find the least positive integer solution of the equation
\begin{equation}\label{eq:solve10}
u^2-2d(2\delta'^2d-1)v^2=1.
\end{equation}
\indent For that purpose we find the continued fraction expansion of the number 
\begin{equation*}
\sqrt{2d(2\delta'^2d-1)}, \enspace\delta'\geq3.
\end{equation*}
We know that
\begin{equation*}
\sqrt{2d(2\delta'^2d-1)}=[a_0; \overline{a_1, a_2, \dots, a_{l-1}, 2a_0}],
\end{equation*} 
where we calculate numbers $a_i$ recursively
\begin{equation*}
a_i=\left\lfloor\frac{s_i+a_0}{t_i}\right\rfloor, \enspace s_{i+1}=a_it_i-s_i, \enspace t_{i+1}=\frac{d-s_{i+1}^2}{t_i}.
\end{equation*}
\indent In our case, we obtain
\begin{equation*}
a_0=\lfloor\sqrt{2d(2\delta'^2d-1)}\rfloor=2d\delta'-1, \enspace s_0=0,\enspace t_0=1;
\end{equation*}
\begin{equation*}
s_1=2d\delta'-1, \enspace t_1=4d\delta'-2d-1, \enspace a_1=1;
\end{equation*}
\begin{equation*}
s_2=2d\delta'-2d, \enspace t_2=2d, \enspace a_2=2\delta'-2;
\end{equation*}
\begin{equation*}
s_3=2d\delta'-2d, \enspace t_3=4d\delta'-2d-1, \enspace a_3=1;
\end{equation*}
\begin{equation*}
s_4=2d\delta'-1, \enspace t_4=1, \enspace a_4=2(2d\delta'-1)=2a_0.
\end{equation*}
\indent We get 
\begin{equation*}
\sqrt{2d(2\delta'^2d-1)}=[2d\delta'-1;\overline{1,2\delta'-2,1,2(2d\delta'-1)}].
\end{equation*}
Now, we can find the least positive integer solution of the equation $(\ref{eq:solve10})$.
Because the length of the period of the expansion is $l=4$, the least positive integer solution of $(\ref{eq:solve10})$ is $(p_3, q_3)$, where numbers $p_i, q_i,\enspace i=0,1,2,3$ are calculated recursively
\begin{equation*}
p_0=a_0, \enspace p_1=a_0a_1+1, \enspace p_k=a_kp_{k-1}+p_{k-2},
\end{equation*}
\begin{equation*}
q_0=1,\enspace q_1=a_1, \enspace q_k=a_kq_{k-1}+q_{k-2},\enspace k=2,3.
\end{equation*}
\indent We obtain
\begin{equation*}
(p_0, q_0)=(2d\delta'-1, 1), \enspace
(p_1, q_1)=(2d\delta', 1),\enspace
(p_2, q_2)=(4\delta'^2d-2d\delta'-1, 2\delta'-1),
\end{equation*}
\begin{equation*}
(p_3, q_3)=(4\delta'^2d-1, 2\delta').
\end{equation*}
\noindent So, the least positive integer solution is $(p_3, q_3)=(u_0, v_0)=(4\delta'^2d-1, 2\delta')$ and we apply Lemma \ref{lemma}.\newline\newline
\indent In our case we have $a=2\delta'^2d-1, \enspace b=2d$. From Lemma \ref{lemma} we get
\begin{equation*}(4\delta'^2d-2)|4\delta'^2d, \enspace 4d|(4\delta'^2d-2).\end{equation*}
\noindent We can easily see that $4d|(4\delta'^2d-2)$ if and only if $4d|2$ which is not possible because $d\in\mathbb{N}$. So, the equation $(\ref{eq:solve9})$ has no solutions. We have proved that there does not exist a positive odd integer $n$ with the property that there exists a pair of positive divisors  $d_1, d_2$ of $\frac{n^2+1}{2}$ such that $d_1+d_2=\delta n$.
\end{proof}
\section*{Acknowledgement} 
We would like to thank Professor Andrej Dujella for many valuable suggestions and a great help with the preparation of this article.
\end{section}

\end{document}